\documentclass{amsart}

\usepackage{amsmath}
\usepackage{amssymb}
\usepackage{mathrsfs}
\usepackage{hyperref}
\usepackage{enumitem}
\usepackage{stmaryrd}
\usepackage[all]{xy}
\usepackage{verbatim}

\newtheorem{Thm}{Theorem}[section]

\newtheorem{Cor}[Thm]{Corollary}
\newtheorem{Def}[Thm]{Definition}

\newtheorem{Eg}[Thm]{Example}
\newtheorem{Hyp}[Thm]{Hypothesis}
\newtheorem{Lem}[Thm]{Lemma}
\newtheorem{Prop}[Thm]{Proposition}
\newtheorem{Qn}[Thm]{Question} 
\newtheorem{Rmk}[Thm]{Remark}

\numberwithin{equation}{Thm}

\newcommand{\FF}{\mathbb{F}}

\newcommand{\QQ}{\mathbb{Q}}

\newcommand{\ZZ}{\mathbb{Z}}

\DeclareMathOperator{\wt}{wt}

\author{Christopher Davis}
\address{University of California, Irvine, Dept of
Mathematics, Irvine, CA 92697}
\email{davis@math.uci.edu}
\date{\today}

\author{Tommy Occhipinti}
\address{University of California, Irvine, Dept of
Mathematics, Irvine, CA 92697}
\email{tocchipi@math.uci.edu}

\title{Which alternating and symmetric groups are unit groups?}

\begin{document}

\begin{abstract}
We prove there is no ring with unit group isomorphic to $S_n$ for $n \geq 5$ and that there is no ring with unit group isomorphic to $A_n$ for $n \geq 5$, $n \neq 8$.  We give examples of rings with unit groups isomorphic to $S_1$, $S_2$, $S_3$, $S_4$, $A_1$, $A_2$, $A_3$, $A_4$, and $A_8$.  We expect our methods to work similarly for other groups with trivial center; in particular, we plan to consider other simple groups in later work.  
\end{abstract}

\maketitle

\section{Introduction}

In this paper we will consider a special case of the general question: For what finite groups $G$ is there a ring with unit group isomorphic to $G$?
We shall see in the following example that this is a nontrivial condition.
\begin{Eg}
There does not exist a ring whose unit group is cyclic of order 5.  The proof is by contradiction.  A ring $R$ such that $R^{\times} \cong C_5$ would have no units of order $2$, and hence $1 = -1$ in $R$.  Thus $R$ is an $\FF_2$-algebra.  By considering the ring homomorphism
\[
\FF_2[x]/(x^5-1) \rightarrow R
\]
which sends $x$ to a generator of $R^{\times}$, and by identifying $\FF_2[x]/(x^5-1)$ with $\FF_2 \times \FF_{2^4}$, we find that $R$ must contain an isomorphic copy of $\FF_{2^4}$.  Hence $R$ has at least $15$ units, and this is a contradiction.
\end{Eg}  

\begin{Rmk}
The general question of which finite groups of odd order occur as the unit group of a ring is answered in \cite{Dit71}.
\end{Rmk}

In the present paper, we determine which symmetric groups and alternating groups are unit groups.  Our proofs are similar in several ways to the above proof.    
For example, although our groups do have elements of order 2 (except in trivial cases), we exploit the fact that our groups have no \emph{central} elements of order 2 (except in trivial cases).  

The main result proved in this paper is the following.

\begin{Thm}
The only finite symmetric groups and alternating groups which are unit groups of rings are the groups
\[
S_1, S_2, S_3, S_4, A_1, A_2, A_3, A_4, A_8.
\]
\end{Thm}

\begin{proof}
The trivial abelian cases of $S_1, S_2$ and $A_1, A_2, A_3$ are treated in Section~\ref{abelian cases}.  The well-known case of $S_3$ is discussed in Section~\ref{S_3 section}.  
An example of a ring with unit group isomorphic to $S_4$ is given in Theorem~\ref{S_4 theorem}.  The fact that $S_n$ does not occur as the unit group of a ring for any $n \geq 5$ is given in Theorem~\ref{unit group S_n}.  The fact that $A_n$ does not occur as the unit group of a ring for any $n \geq 5, n \neq 8$ is given in Theorem~\ref{unit group A_n}.    Two examples of rings with unit group isomorphic to $A_4$ are given in Section~\ref{A_4 case}. The classical result that $M_{4 \times 4}(\FF_2)$ has unit group isomorphic to $A_8$ is recalled in Theorem~\ref{A_8 theorem}.
\end{proof}

\subsection*{Notation and conventions}  Our rings are assumed unital but not necessarily commutative, and ring homomorphisms send $1$ to $1$.  Also, when we say $S$ is a subring of $R$, we include the assumption that 1 is the same in both rings.  For a ring $R$, we let $R^{\times}$ denote the unit group of $R$.  The groups $G$ considered in this paper will be finite.  For a group $G$ we let $Z(G)$ denote its center.   For $T$ a subset of a group $G$, we also write $T$ for the element of the group algebra
\[
\sum_{t \in T} t \in R[G].
\]
(Here $R$ will be understood from context; for us, $R$ is typically $\FF_2$.)  We also write $\langle T \rangle$ for the subgroup of $G$ generated by $T$.  We write $\iota$ for the identity element of $A_n$ or $S_n$.  When we discuss a normalizer $N_G(T)$ or a centralizer $Z_G(T)$, we do not necessarily assume that $T$ is a subgroup.  For example, $N_G(T)$ is the set of $g \in G$ such that $gTg^{-1} = T$; in particular, it is not necessarily the same as the normalizer of $\langle T \rangle$.  We write $D_n$ for the dihedral group of order $2n$.

\section{Unit groups with trivial center} \label{general section}

In this section, we describe some general results which will be applied to the special cases of alternating groups and symmetric groups in the following sections.  We hope to apply the results of this section to other groups in subsequent work.  Our motivating question is the following.  

\begin{Qn} \label{G question}
Let $G$ denote a group with trivial center.  Does there exist a ring with unit group isomorphic to $G$?
\end{Qn}

We begin with an easy exercise.

\begin{Prop} \label{trivial center implies characteristic 2}
Let $G$ denote a finite group with trivial center, and let $R$ denote a ring with unit group $R^{\times} \cong G$.  Then $R$ has characteristic 2.
\end{Prop}

\begin{proof}
The elements $1$ and $-1$ are units in $R$ and are in the center of $R$, hence are in the center of $R^{\times}$.  Hence $1 = -1$.
\end{proof}

The following reduces our Question~\ref{G question} into a question about finite rings.

\begin{Prop} \label{ideal exists}
Let $G$ denote a finite group with trivial center.
If there exists a ring with unit group isomorphic to $G$, then there exists a two-sided ideal $I \subseteq \FF_2[G]$ such that the quotient $\FF_2[G]/I$ has unit group isomorphic to $G$, and furthermore such that the natural composition
\[
G \subseteq \FF_2[G]^{\times} \rightarrow  (\FF_2[G]/I)^{\times} \cong G
\]
is the identity map.
\end{Prop}

\begin{proof}
Let $R$ denote a ring with unit group isomorphic to $G$, and fix an isomorphism $R^{\times} \cong G$.
There exists a unique homomorphism 
\[
\varphi: \ZZ[G] \rightarrow R,
\]
such that the induced map
\[
\varphi: G \rightarrow \ZZ[G]^{\times} \rightarrow R^{\times} \cong G
\]
is the identity map. Because $G$ has trivial center, by Proposition~\ref{trivial center implies characteristic 2}, we know $R$ has characteristic $2$.  Hence our homomorphism $\varphi$ factors through a homomorphism
\[
\varphi: \FF_2[G] \rightarrow R.
\]
Let $R'$ denote the image of $\varphi$.  Because $R'$ is a subring of $R$, we know that the unit group of $R'$ is a subgroup of $G$.  On the other hand, we checked above that the image of $\varphi$ contains $G$.  Hence the unit group of $R'$ is equal to $G$.  This completes the proof.
\end{proof}

Our approach to Question~\ref{G question} will be to consider the restrictions on an ideal $I \subseteq \FF_2[G]$ as described in Proposition~\ref{ideal exists}. 

\begin{Hyp}
Throughout this section, let $G$ denote a finite group with trivial center and let $I$ denote an ideal as in Proposition~\ref{ideal exists}.  We also write $\varphi$ for the natural map $\FF_2[G] \rightarrow \FF_2[G]/I$.  
\end{Hyp}

\begin{Def}
The \emph{weight} of an element $x \in \FF_2[G]$ is the number of non-zero coefficients that appear in the expression
\[
x = \sum_{g \in G} a_{g} g \qquad (a_{g} \in \FF_2).
\]
\end{Def}

\begin{Lem} \label{weight 2 is impossible}
The ideal $I$ contains no elements of weight 2.
\end{Lem}

\begin{proof}
We prove that if $g + h \in I$, then $g = h$; this implies that $I$ contains no weight 2 elements.
If $g + h \in I$, then $\varphi(g) = - \varphi(h)$.  Because our ring is characteristic $2$, this implies $\varphi(g) = \varphi(h)$.  Because we have assumed that the restriction of $\varphi$ to $G$ is injective, this can only happen if $g = h$.
\end{proof}

\begin{Lem} \label{congruent to sigma}
Let $x \in \FF_2[G]$ denote a unit.  Then there exists $\sigma_x \in G$ such that $x + \sigma_x \in I$.
\end{Lem}

\begin{proof}
This follows from the following remarks:
\begin{itemize}
\item $\varphi(x)$ is a unit and so $\varphi(x) = \varphi(\sigma_x)$ for some $\sigma_x \in G$;
\item $x \equiv \sigma_x \bmod I$ for some $\sigma_x \in G$;
\item $x + \sigma_x = x - \sigma_x$ because our ring is characteristic 2.
\end{itemize}
\end{proof}

\begin{Prop} \label{centralizer of normalizer}
Let $T \in \FF_2[G]$ denote a unit, which we identify with a subset $T \subseteq G$.  The element $\sigma_T$ described in Lemma~\ref{congruent to sigma} is in the centralizer of $N_{G}(T)$ in $G$, where $N_{G}(T)$ is the normalizer of $T$ in $G$.  
\end{Prop}

\begin{proof}
Because $I$ is a two-sided ideal, for any $g \in G$ we have
\begin{align*}
g(T + \sigma_T) &\in I \\
(T + \sigma_T) g &\in I.\\
\intertext{In particular, taking $g \in N_{G}(T)$ and adding these last two elements, we find}
g T + Tg + g\sigma_T + \sigma_T g &\in I\\
g\sigma_T + \sigma_T g &\in I.
\end{align*}
By Lemma \ref{weight 2 is impossible}, the elements $g\sigma_T$ and $\sigma_T g$ cannot be distinct elements of $G$.  Hence $g$ and $\sigma_T$ commute.  Because $g \in N_{G}(T)$ was arbitrary, we deduce that $\sigma_T$ is in the centralizer of $N_{G}(T)$ in $G$, as required.
\end{proof}

We are now ready to apply these general results to some specific groups.

\section{An example: unit group $S_3$} \label{S_3 section} 

There is a well-known ring with unit group isomorphic to $S_3$, namely, the matrix ring $M_{2 \times 2}(\FF_2)$.  In this section, we apply the general techniques of the previous section to the group $S_3$ as a way of illustrating our approach.  

The symmetric group $S_3$ has trivial center, and so the results of Section~\ref{general section} all apply in the case $G \cong S_3$.  We consider the restrictions on an ideal $I \subseteq \FF_2[S_3]$ such that 
\[
(\FF_2[S_3]/I)^{\times} \cong S_3,
\]
and such that furthermore the induced map 
\[
S_3 \rightarrow \FF_2[S_3]^{\times} \rightarrow (\FF_2[S_3]/I)^{\times} \cong S_3
\] 
is the identity map.  

Consider the element 
\[
H_1 := \sum_{\sigma \in S_3} \sigma \in \FF_2[S_3]
\]
corresponding to the full subgroup $S_3$.  It is easy to check that $H_1^2 = 0$ and that $(H_1 + \iota)^2 = \iota$.  Hence $H_1 + \iota$ is a unit in $\FF_2[S_3]$; abbreviate this unit by $T$.  By Lemma~\ref{congruent to sigma}, there must exist an element $\sigma_T \in S_3$ such that $T+ \sigma_T \in I$.  Because the normalizer of $T = S_{3} \setminus \{\iota\}$ in $S_3$ is the full group $S_3$, by Proposition~\ref{centralizer of normalizer}, we must have $\sigma_{T} = \iota$, and hence $T + \iota \in I$, and hence $H_1 \in I$.  
The reader may check that the $32$-element ring
$\FF_2[S_3]/(H_1)$
has unit group isomorphic to $S_3$.  

Let $\tau \in S_3$ denote a $3$-cycle, and let $H_2 := \iota + \tau + \tau^2$.  Then $(H_2) = (H_1, H_2)$, and the reader may check that the $16$-element ring $\FF_2[S_3]/(H_2)$ is isomorphic to $M_{2 \times 2}(\FF_2)$.  Hence $\FF_2[S_3]/(H_2)$ is another example of a ring with unit group isomorphic to $S_3$.

\section{Unit group $S_n$}

Having analyzed the case of $S_3$ in the previous section, we postpone the case of $S_4$ and turn our attention to $S_n$ for $n \geq 5$.  These groups have trivial center, so again the results of Section~\ref{general section} apply.  Our goal is to prove the following theorem.  

\begin{Thm} \label{unit group S_n}
There does not exist a ring with unit group isomorphic to $S_n$ for any $n \geq 5$.  
\end{Thm}

\begin{proof}
By way of contradiction, we suppose that we have a ring with unit group isomorphic to $S_n$.  Let $I \subseteq \FF_2[S_n]$ denote an ideal satisfying the hypotheses of Proposition~\ref{ideal exists}.  Our goal is to produce an element of weight 2 in the ideal $I$ and thus reach a contradiction.  

Let $\tau = (12345)$ and consider the element $T := \iota + \tau^2 + \tau^3 \in \FF_2[S_n]$.  The fact that $T$ is a unit of order 3 and with inverse $1 + \tau + \tau^4$ is readily verified\footnote{The existence of such an order 3 unit T is explained as follows.  By the Chinese Remainder Theorem,  $\FF_2[\tau] \cong \FF_2 \times \FF_{2^4}$.  The unit group of $\FF_{2^4}$ is cyclic of order 15 and hence $\FF_2[\tau]^{\times}$ has a cyclic subgroup of order 3.}.   
By Lemma~\ref{congruent to sigma}, there exists some $\sigma \in S_n$ such that $\iota + \tau^2 + \tau^3 + \sigma \in I$.  By Proposition~\ref{centralizer of normalizer}, the element $\sigma$ must be in the centralizer of the normalizer of $\{\iota, \tau^2, \tau^3\}$ in $S_n$.  One may check that the normalizer of $\{\iota, \tau^2, \tau^3\}$ in $S_n$ is $D_{5} \times S_{n-5}$ and that the centralizer of $D_{5} \times S_{n-5}$ is $Z(S_{n-5})$.  

Thus $\sigma \in Z(S_{n-5})$.  If $\sigma = \iota$, then $\iota + \tau^2 + \tau^3 + \iota = \tau^2 + \tau^3 \in I$ is a weight 2 element in $I$, which is not allowed.  The only remaining case is $n = 7$ and $\sigma = (67)$.  
Let $T = \iota + \tau^2 + \tau^3 + \sigma \in I$.  Raising both sides to the $16$-th power, we find that
\[
T^{16} = \iota^{16} + (\tau^2)^{16} + (\tau^3)^{16} + \sigma^{16} \in I.
\]
(We used here that $\tau$ and $\sigma$ commute, and that our base ring has characteristic 2.)
Because $\tau$ has order five and $\sigma$ has order two, we find
\[
T^{16} = \iota + \tau^2 + \tau^3 + \iota = \tau^2 + \tau^3 \in I, 
\]
which is a contradiction.  This completes the proof that there are no rings with unit group isomorphic to $S_n$, for $n \geq 5$.  
\end{proof}

\section{Unit group $A_n$}

The methods of the previous section carry over directly to the case of the alternating groups $A_n$.  The only substantive difference is that our proof breaks down in the case $A_8$, essentially because $A_{8-5} = A_3$ is abelian.  This is to be expected, though, because as we will see in Theorem~\ref{A_8 theorem}, the ring $M_{4 \times 4}(\FF_2)$ has unit group isomorphic to $A_8$.

\begin{Thm} \label{unit group A_n}
There does not exist a ring with unit group isomorphic to $A_n$ for any $n \geq 5, n \neq 8$.  
\end{Thm}

\begin{proof}
The proof is very similar to the proof of Theorem~\ref{unit group S_n}, so we focus only on the main steps.  Let $I \subseteq \FF_2[A_n]$ denote an ideal satisfying the hypotheses of Proposition~\ref{ideal exists}.  Because the $5$-cycle $\tau = (12345)$ is in $A_n$ for any $n \geq 5$, we again have a unit $\iota + \tau^2 + \tau^3$, and we again wish to consider possible values of $\sigma \in A_n$ such that $\iota + \tau^2 + \tau^3 + \sigma \in I$.  
One may check that the normalizer of $\{\iota, \tau^2, \tau^3\}$ in $A_n$ is $D_5 \times A_{n-5}$.  If $n \geq 5, n \neq 8$, then the centralizer of this subgroup in $A_n$ is trivial, and hence $\sigma = \iota$, and we are finished as before.  
\end{proof}

\begin{Rmk}
If $n = 8$, then the element $\sigma$ described in the previous proof should be in the centralizer of $D_5 \times A_3$; this centralizer is a cyclic group of order $3$.  In the proof of Theorem~\ref{unit group S_n}, we at one point considered $\sigma^{16}$.  In the $S_n$ case, we were able to prove that $\sigma^{16}$ was always trivial.  In the $A_8$ case, $\sigma$ may have order $3$, and so the proof breaks down, as it should because $\left(M_{4 \times 4}(\FF_2)\right)^{\times} \cong A_8$; see Theorem~\ref{A_8 theorem} below.  
\end{Rmk}

\section{Unit group $S_4$}

The only remaining nonabelian symmetric group to consider is $S_4$.  We describe rings with unit group isomorphic to $S_4$ in this section.  We first need some results similar to the results in Section~\ref{general section}.

\begin{Lem}
Let $H \subseteq S_n$ denote a subgroup of even order.  Then $H^2 = 0 \in \FF_2[S_n]$.  (Recall our convention that we write $H$ for the element $\sum_{h \in H} h \in \FF_2[S_n]$.)
\end{Lem}

\begin{proof}
We have
\[
H^2 = |H| \cdot H = 0,
\]
because $|H|$ is even.
\end{proof}

\begin{Lem} \label{even unit lemma}
Let $H \leq S_n$ denote a subgroup of even order as above.  The element $H + \iota$ is a unit in $\FF_2[S_n]$.  
\end{Lem}

\begin{proof}
One checks that $(H + \iota)^2 = \iota$.
\end{proof}

\begin{Prop} \label{S_4 ideal}
Let $R$ denote a ring with unit group isomorphic to $S_4$, and let $I \subseteq \FF_2[S_4]$ denote an ideal as in Proposition~\ref{ideal exists}.  
\begin{enumerate}
\item \label{S_3 part} The ideal $I$ contains an isomorphic copy of $S_3$ (and hence all 4 isomorphic copies of $S_3$) inside of $S_4$.  
\item \label{weight 3 part} The ideal $I$ contains either
\[
\iota + (24) + (12)(34) + (1234)
\]
or
\[
\iota + (24) + (12)(34) + (1432).
\]
\end{enumerate}
\end{Prop}

\begin{proof}
To prove (\ref{S_3 part}), let $H$ denote an isomorphic copy of $S_3$ contained inside of $S_4$, and view $H$ as an element of $\FF_2[S_4]$ as usual.  Then by Lemma~\ref{even unit lemma}, $H + \iota$ is a unit in $\FF_2[S_4]$.  Then by Proposition~\ref{centralizer of normalizer}, we find that 
\[
H + \iota + \sigma \in I
\]
for some $\sigma$ in the centralizer of $H$ in $S_4$.  The only possibility is $\sigma = \iota$, which completes the proof of (\ref{S_3 part}).  

To prove (\ref{weight 3 part}), we again find a unit $T \in \FF_2[S_4]$ and consider the possible values of $\sigma$ such that $T + \sigma \in I$.  Let $T = \iota + (24) + (12)(34)$.  The fact that $T$ is a unit of order 4 with inverse 
\[
\iota + (1234) + (1432) + (14)(23) + (13)
\]
is readily verified\footnote{The unit $T$ was found using \cite[Theorem~1.2]{Jen41}, which shows that $(24)+(12)(34)$ is nilpotent, because it is an even weight element consisting of elements in a copy of the $2$-group $D_4 \subset S_4$.}.   
Using Magma, it was verified that $\sigma = (1234)$ and $\sigma = (1432)$ were the only choices for which the two-sided ideal generated by $T + \sigma$ did not contain an element of weight 2.  
\end{proof}

\begin{Thm} \label{S_4 theorem}
Let $J_1$ (respectively, $J_2$) denote the two-sided ideal in $\FF_2[S_4]$ generated by the two elements
\[
\iota + (24) + (12)(34) + (1234) \qquad \text{  (respectively, } \iota + (24) + (12)(34) + (1432))
\] 
and 
\[
\iota + (12) + (23) + (13) + (123) +(132).
\]
Let $R_1 := \FF_2[S_4]/J_1$ and let $R_2 := \FF_2[S_4]/J_2$.

The rings $R_1$ and $R_2$ are nonisomorphic rings with $128$ elements and with unit group isomorphic to $S_4$.  Every ring with unit group isomorphic to $S_4$ contains a subring isomorphic to either $R_1$ or $R_2$.
\end{Thm}

\begin{proof}
It can be verified in Magma that $R_1$ is a ring with $128$ elements and with exactly 24 distinct units corresponding to the cosets $\sigma + J_1$, for $\sigma \in S_4$.  (Sample Magma code which verifies this claim is given in Appendix~\ref{appendix code}.)  The same can be done for $R_2$, or it can be checked that $R_2 \cong R_1^{\mathrm{op}}$, from which the claim that $R_2^{\times} \cong R_1^{\times}$ follows.  

We next check that $R_1$ and $R_2$ are not isomorphic.  An isomorphism $\psi: R_1 \rightarrow R_2$ would induce an isomorphism $\psi: R_1^{\times} \rightarrow R_2^{\times}$.  Because the only automorphisms of $S_4$ are inner automorphisms, the restriction of $\psi$ to $S_4$ would have to correspond to conjugation by some element $\tau \in S_4$.  Consider the image of an arbitrary element $x := \sigma_1 + \cdots + \sigma_n \in J_1$ under the composition 
\[
\FF_2[S_4] \rightarrow \FF_2[S_4]/J_1 \stackrel{\psi}{\rightarrow} \FF_2[S_4]/J_2.
\]  
On one hand, $x$ must map to the coset $J_2$.  On the other hand, $x$ must map to  
\[
\tau \sigma_1 \tau^{-1} + \cdots + \tau \sigma_n \tau^{-1} + J_2.
\]

In this way, we reduce to showing that there does not exist an element $\tau \in S_4$ such that $\tau J_1 \tau^{-1} = J_2$.  Because $J_1$ is a two-sided ideal, we know that $\tau J_1 \tau^{-1} \subseteq J_1$, so we have reduced to showing $J_2 \not\subset J_1$.  If $J_2$ were contained in $J_1$, then the weight two element 
\[
\Big(\iota + (24) + (12)(34) + (1234)\Big)+  \Big(  \iota + (24) + (12)(34) + (1432) \Big) =  (1234) + (1432)
\]
would be contained in $J_1$, which would contradict the fact that the cosets $(1234) + J_1$ and $(1432) + J_1$ are distinct.

We now prove the final assertion, that any ring $R$ with unit group isomorphic to $S_4$ contains a subring isomorphic to $R_1$ or $R_2$.  We know that such a ring $R$ contains as a subring $\FF_2[S_4]/I$, where $I$ is an ideal as in Proposition~\ref{ideal exists}.   So it suffices to show that if $I$ is an ideal as in Proposition~\ref{ideal exists}, then $I = J_1$ or $J_2$.  It was proven in Proposition~\ref{S_4 ideal} that $I$ must contain either $J_1$ or $J_2$.  So it remains only to show that the ideal $I$ cannot be strictly larger than $J_1$ or $J_2$.  It was verified in Magma that for each nonzero principal ideal $(x)$ in $R_1$, the ring $R_1/(x)$ has at most 6 units.
\end{proof}

\section{The remaining cases}

\subsection{The abelian cases}  \label{abelian cases} These cases are trivial, but we include them for the sake of completeness.  

\begin{Prop}
For each group $G$ in the list
\[
S_1, S_2, A_1, A_2, A_3,
\]
there exists a ring with unit group isomorphic to $G$.
\end{Prop}

\begin{proof}
The groups $S_1, S_2, A_1, A_2, A_3$ are cyclic groups of order $1,2,1,1,3$, respectively.  Hence, they are isomorphic to the unit groups of the fields $\FF_2, \FF_3, \FF_2, \FF_2, \FF_4$, respectively.  
\end{proof}

\subsection{Unit group $A_4$} \label{A_4 case}

In this section we give two different rings with unit group isomorphic to $A_4$.  We describe the first ring as an explicit quotient of $\FF_2[A_4]$.   We describe the second ring as a quotient of the ring of Hurwitz quaternions.  
\begin{Thm} \label{A_4 theorem}
Let $J \subseteq \FF_2[A_4]$ denote the two-sided ideal generated by the elements
\[
\iota + (12)(34) + (13)(24) + (14)(23) 
\]
and 
\[
\iota + (132) + (12)(34) + (143).
\]
Then the quotient $\FF_2[A_4]/J$ is a ring with 32 elements and with unit group isomorphic to $A_4$.  
\end{Thm}

\begin{proof}
By adapting the Magma code in Appendix~\ref{appendix code}, this assertion is readily verified. 
\end{proof}

We next use quaternions to give a second example of a ring with unit group isomorphic to $A_4$.  First we set some notation.  

\begin{Def} \label{quaternion notation}
Let $B$ denote the division algebra $\QQ + \QQ i + \QQ j + \QQ k$, where $i,j,k$ are defined as in the Hamilton quaternions.  Let $\omega = \frac{1 + i + j + k}{2}$ and let $\mathcal{O} \subset B$ denote 
\[
\ZZ \oplus \ZZ i \oplus \ZZ j \oplus \ZZ \omega \subseteq B;
\]
this subgroup $\mathcal{O} \subset B$ is a subring known as the \emph{Hurwitz quaternions}.    
\end{Def}

The authors thank Noam Elkies for the following example.  

\begin{Thm} \label{A_4 theorem 2}
Let $\mathcal{O}$ denote the ring of Hurwitz quaternions, as in Definition~\ref{quaternion notation}.  The quotient ring $\mathcal{O}/2\mathcal{O}$ is a ring with 16 elements and with unit group isomorphic to $A_4$.    
\end{Thm}

\begin{proof}
By \cite[Proposition~3]{HM04}, the unit group $\mathcal{O}^{\times}$ is isomorphic to the binary tetrahedral group; in particular, there is a short exact sequence 
\[
1 \to \{ \pm 1 \} \rightarrow \mathcal{O}^{\times}  \to  A_4 \rightarrow 1.
\]
The kernel of the induced map
\[
\mathcal{O}^{\times} \rightarrow \left( \mathcal{O}/2\mathcal{O}\right)^{\times}
\]
is exactly $\mathcal{O}^{\times} \cap (1 + 2\mathcal{O}) = \{\pm 1\}$.  Hence $(\mathcal{O}/2\mathcal{O})^{\times}$ contains a subgroup isomorphic to $A_4$.  On the other hand, $\mathcal{O}/2\mathcal{O}$ is a ring with 16 elements.  Hence its unit group must be precisely $A_4$.  
\end{proof}

\begin{Rmk}
The ring $\mathcal{O}/2\mathcal{O}$ from Theorem $\ref{A_4 theorem 2}$ is isomorphic to $\FF_2[A_4]/J$, where $J$ is the ideal generated by $\iota + (123) + (132)$.  
\end{Rmk}

\subsection{Unit group $A_8$}  The only remaining case is $A_8$, which we recall in the following theorem.

\begin{Thm} \label{A_8 theorem}
The unit group of $M_{4\times 4}(\FF_2)$ is isomorphic to $A_8$.  
\end{Thm}

\begin{proof}
We have 
\[
M_{4 \times 4}(\FF_2)^{\times} = GL_4(\FF_2) = PSL_4(\FF_2) \cong A_8.
\]
For this last isomorphism, see \cite[Section~3.12.1]{Wil09}.
\end{proof}

\subsection*{Acknowledgments} 

The question studied in this paper was first posed to the second author by Charles Toll.  Special thanks to him, and to John F.\ Dillon, Dennis Eichhorn, Noam Elkies, Kiran Kedlaya, and Ryan Vinroot for many useful conversations.  The authors made frequent use of both Magma \cite{Magma} and Sage \cite{Sage} while investigating this question.  The free Magma online calculator \url{http://magma.maths.usyd.edu.au/calc/} was especially helpful.

\appendix

\section{Sample Magma code} \label{appendix code}

To find a ring with unit group isomorphic to $S_4$, we explicitly computed the unit group of a certain quotient of $\FF_2[S_4]$.  The computation was done in Magma, and we next provide sample code which performs this computation.  

\begin{Eg}  The following was used at the beginning of the proof of Theorem~\ref{S_4 theorem}.  It first creates the ring $R_1$ and counts its total number of elements as well as its number of units.  It then ensures that no elements $\sigma_1 \neq \sigma_2$ become equal in $R_1 \cong \FF_2[S_4]/I$.   

\bigskip

\begin{verbatim}
G:=SymmetricGroup(4);
F2G:=GroupAlgebra(GF(2), G);

x1:= F2G!G!1+F2G!G!(2,4)+F2G!G!(1,2)(3,4)+F2G!G!(1,2,3,4);

x2:=F2G!0;
H1:=sub<G|(1,2),(1,2,3)>;
for h in H1 do
    x2:=x2+F2G!h;
end for;

I:=ideal<F2G|x1, x2>;
R1:= F2G/I;

numunits:=0;
for x in R1 do
    if IsUnit(x) then
        numunits:=numunits+1;
    end if;
end for;

#(F2G/I);
numunits;

for y1 in G do
    for y2 in G do
        if F2G!y1 + F2G!y2 in I then
            if y1 ne y2 then
                y1;
                y2;
            end if;
        end if;
    end for;
end for;
\end{verbatim}
\end{Eg}

\bibliography{units}

\begin{thebibliography}{1}

\bibitem{Magma}
Wieb Bosma, John Cannon, and Catherine Playoust.
\newblock The {M}agma algebra system. {I}. {T}he user language.
\newblock {\em J. Symbolic Comput.}, 24(3-4):235--265, 1997.
\newblock Computational algebra and number theory (London, 1993).

\bibitem{Dit71}
S.~Z. Ditor.
\newblock On the group of units of a ring.
\newblock {\em Amer. Math. Monthly}, 78:522--523, 1971.

\bibitem{HM04}
Emmanuel Hallouin and Christian Maire.
\newblock Cancellation in totally definite quaternion algebras.
\newblock {\em J. Reine Angew. Math.}, 595:189--213, 2006.

\bibitem{Jen41}
S.~A. Jennings.
\newblock The structure of the group ring of a {$p$}-group over a modular
  field.
\newblock {\em Trans. Amer. Math. Soc.}, 50:175--185, 1941.

\bibitem{Sage}
W.\thinspace{}A. Stein et~al.
\newblock {\em {S}age {M}athematics {S}oftware ({V}ersion 5.3)}.
\newblock The Sage Development Team, 2012.
\newblock {\tt http://www.sagemath.org}.

\bibitem{Wil09}
R.~Wilson.
\newblock {\em The Finite Simple Groups}.
\newblock Graduate Texts in Mathematics. Springer, 2009.

\end{thebibliography}
\bibliographystyle{plain}

\end{document}